\documentclass[12pt,reqno]{amsart}
\usepackage{amssymb}
\usepackage{graphicx}
\usepackage{xcolor}

\usepackage[all]{xy}
\DeclareFontFamily{U}{mathb}{\hyphenchar\font45}
\DeclareFontShape{U}{mathb}{m}{n}{
      <5> <6> <7> <8> <9> <10> gen * mathb
      <10.95> mathb10 <12> <14.4> <17.28> <20.74> <24.88> mathb12
      }{}
\DeclareSymbolFont{mathb}{U}{mathb}{m}{n}
\DeclareMathSymbol{\righttoleftarrow}{3}{mathb}{"FD}

\theoremstyle{plain}
\newtheorem{prop}{Proposition}[section]
\newtheorem{theo}[prop]{Theorem}

\newtheorem{lemm}[prop]{Lemma}
\theoremstyle{remark}
\newtheorem{rema}[prop]{Remark}

\theoremstyle{definition}
\newtheorem{defi}[prop]{Definition}

\newtheorem{exam}[prop]{Example}
\numberwithin{equation}{section}

\newcommand{\bP}{{\mathbb P}}

\newcommand{\G}{{\mathbb G}}
\newcommand{\N}{{\mathbb N}}

\newcommand{\cB}{{\mathcal B}}

\newcommand{\cO}{{\mathcal O}}
\newcommand{\cI}{{\mathcal I}}

\newcommand{\fA}{{\mathfrak A}}
\newcommand{\sD}{{\mathfrak D}}

\newcommand{\fs}{{\mathfrak s}}
\newcommand{\BC}{{\mathcal B}{\mathcal C}}

\newcommand{\rH}{{\mathrm H}}

\newcommand{\ra}{\rightarrow}

\newcommand{\bZ}{{\mathbb Z}}

\newcommand{\fS}{{\mathfrak S}}

\newcommand{\eqto}{\stackrel{\lower1.5pt\hbox{$\scriptstyle\sim\,$}}\to}
\newcommand{\eqdashto}{\stackrel{\lower1.5pt\hbox{$\scriptstyle\sim\,$}}\dashrightarrow}
\newcommand{\actsfromleft}{\mathrel{\reflectbox{$\righttoleftarrow$}}}
\newcommand{\actsfromright}{\righttoleftarrow}
\DeclareMathOperator{\Alg}{Alg}
\DeclareMathOperator{\Gal}{Gal}

\DeclareMathOperator{\Pic}{Pic}
\DeclareMathOperator{\Spec}{Spec}

\DeclareMathOperator{\Hom}{Hom}

\DeclareMathOperator{\Burn}{Burn}
\DeclareMathOperator{\Bir}{Bir}

\DeclareMathOperator{\Ind}{Ind}

\begin{document}
\title[Equivariant Burnside groups]{Equivariant Burnside groups: Structure and operations}

\author{Andrew Kresch}
\address{
  Institut f\"ur Mathematik,
  Universit\"at Z\"urich,
  Winterthurerstrasse 190,
  CH-8057 Z\"urich, Switzerland
}
\email{andrew.kresch@math.uzh.ch}
\author{Yuri Tschinkel}
\address{
  Courant Institute,
  251 Mercer Street,
  New York, NY 10012, USA
}

\email{tschinkel@cims.nyu.edu}

\address{Simons Foundation\\
160 Fifth Avenue\\
New York, NY 10010\\
USA}

\date{May 6, 2021}

\begin{abstract}
We introduce and study functorial and combinatorial constructions concerning equivariant Burnside groups.
\end{abstract}

\maketitle

\section{Introduction}
\label{sec.intro}

Let $G$ be a finite group, and $k$ a field of characteristic zero containing all 
roots of unity of order dividing $G$. 
In this paper, we continue the study of a new invariant in $G$-equivariant birational 
geometry over $k$, the {\em equivariant Burnside group}
$$
\Burn_n(G),
$$
introduced in \cite{BnG},  and building on \cite{KT}, \cite{kontsevichpestuntschinkel}, \cite{Bbar}, and 
\cite{HKTsmall}.

The class of an $n$-dimensional $G$-variety in this group 
is computed on an appropriate smooth $G$-birational model $X$, 
called {\em standard form}: 
after a sequence of $G$-equivariant blowups we may assume that 
\cite{reichsteinyoussinessential}: 
\begin{itemize}
\item there exists a Zariski open $U\subset X$ such that the $G$-action on $U$ is free, 
\item the complement $X\setminus U$ is a $G$-invariant simple normal crossing divisor, 
\item for every $g\in G$ and every irreducible component $D$ of $X\setminus U$, either $g(D)=D$
or $g(D)\cap D = \emptyset$. 
\end{itemize}
The standard form is preserved under $G$-equivariant blowups with smooth centers which have normal crossings with respect to the components of $D$.
Moreover, 
the stabilizer of every $x$ on such $X$ is an {\em abelian} subgroup of $G$ \cite[Thm. 4.1]{reichsteinyoussinessential}.
On such a model, the class of $X\actsfromright G$ is defined by: 
\begin{equation}
\label{eqn:class}
[X\actsfromright G] :=\sum_{H\subseteq G} \sum_F \fs_F \,
 \in \Burn_n(G),
\end{equation}
with summation over conjugacy classes of {\em abelian} subgroups   $H\subseteq G$ and strata $F\subseteq X$ with 
{\em generic} stabilizer $H$; the symbol 
\begin{equation}
\label{eqn:symb}
\fs_F:=(H,N_G(H)/H\actsfromleft k(F),\beta_F(X))
\end{equation}
records the action of the normalizer $N_G(H)$ of $H$ 
on $k(F)$, the product of the function fields of the components of $F$, as well as the generic normal bundle representation $\beta_F(X)$ of $H$. 
The class in \eqref{eqn:class} 
takes values in a quotient of the free abelian group generated by such symbols, 
by certain {\em blow-up} relations, spelled out in \cite[Definition 4.2]{BnG}, 
and ensuring that this expression is a well-defined 
$G$-birational invariant \cite[Theorem 5.1]{BnG}.

In \cite{HKTsmall}, we presented first geometric applications of this invariant. 
Here, we continue to explore {\em functorial} and {\em combinatorial} properties of $\Burn_n(G)$. 
We introduce and study:
\begin{itemize}
\item filtrations on $\Burn_n(G)$,  
\item the restriction homomorphism 
$$
\Burn_n(G)\to \Burn_n(G'),
$$
where $G'\subset G$ is any subgroup,
\item products, 
\item 
a combinatorial analog $\BC_n(G)$ of $\Burn_n(G)$,  obtained by  
forgetting {\em field-theoretic} information, while keeping only discrete invariants encoded in a symbol \eqref{eqn:symb}. 
\end{itemize}
One of the motivating problems in this field is to distinguish equivariant birational types of (projectivizations of) linear actions (see, e.g., \cite{vinberg}, \cite{katsylo}). A sample question, raised in \cite[Section 8]{DolIsk}, is: {\it Are there 
isomorphic finite subgroups of $\mathrm{PGL}_3$ which are not conjugate in the plane Cremona group?} 

Examples of equivariantly nonbirational representations 
considered in \cite{reichsteinyoussininvariant}
required that $G$ contains an abelian $p$-subgroup of rank
equal to the dimension of the representation. 
Our formalism yields new examples 
without the rank condition on the group, see Example~\ref{exa.Z5S3}.

\

\medskip
\noindent
\textbf{Acknowledgments:}
We are very grateful to Brendan Hassett for his interest  and help on this and related projects. 
The first author was partially supported by the
Swiss National Science Foundation. 
The second author was partially supported by NSF grant 2000099.

\section{Generalities}
\label{sect:gen}

We adopt notational conventions from \cite{BnG}: 
\begin{itemize}
\item $G$ is a finite group, 
\item $k$ is a field of characteristic 0, containing roots of unity of order $|G|$, 
\item $H\subseteq G$ is an abelian subgroup, with character group 
$$
H^\vee:=\Hom(H,k^\times),
$$
\item $\Bir_d(k)$ is the set of birational equivalence classes of $d$-dimensional algebraic varieties over $k$ , i.e., 
the set of finitely generated fields of transcendence degree $d$ over $k$; we identify a field with its isomorphism class in $\Bir_d(k)$, 
\item $\Alg_N(K_0)$ is the set of isomorphism classes of Galois algebras $K$ over $K_0\in \Bir_d(k)$
for the group
\[ N:=N_G(H)/H, \]
satisfying \\
\

{\bf Assumption 1}:  the composite homomorphism
\begin{equation}
\label{eqn.weak}
\rH^1(N_G(H),K^\times)\to \rH^1(H,K^\times)^N\to H^\vee
\end{equation}
is surjective (see \cite[Section 2]{BnG} for more details).
\item More generally, for a subgroup $M\subset N$ we denote by
$\Alg_M(K_0)$ the set of isomorphism classes of $M$-Galois algebras $K/K_0$
(i.e., Galois algebras $K$ over $K_0$ for the group $M$), such that 
$\Ind_M^N(K)$ 
satisfies
Assumption 1. Of particular interest is 
\[ 
Z:=Z_G(H)/H\subseteq N=N_G(H)/H.
\]
\end{itemize}

\begin{lemm}
\label{lem.Z}
Let $K\in \Alg_N(K_0)$. 
Then
$$
K\cong \Ind_{Z}^{N}(K')
$$ 
for some $K'\in \Alg_Z(K_0)$.
\end{lemm}

\begin{proof}
With notation of Assumption 1, we have
trivial $H$-action on 
$K^\times$, so 
$$
\mathrm{H}^1(H,K^\times)=\Hom(H,K^\times).
$$
Writing $K\cong K^1\times\dots\times  K^\ell$, where each $K^i$ is a field,
as rightmost map in \eqref{eqn.weak} we take
$$
\mathrm{H}^1(H,K^\times)\to\mathrm{H}^1(H,(K^1)^\times)\cong H^\vee.
$$
Projection $K^\times\to (K^1)^\times$ is equivariant for the
subgroup $Y\subseteq N$, defined by the
condition of sending $K^1$ to $K^1$, where the action on
$\mathrm{H}^1(H,(K^1)^\times)$ is given just by conjugation on $H$.
Assumption 1 implies that the conjugation action is trivial,
i.e., $Y\subseteq Z$.
So the result holds with 
$$
K'=\Ind_Y^Z(K^1).
$$
\end{proof}

\begin{rema}
\label{rem.groupH2}
Assumption 1, for an $N$-Galois algebra $K/K_0$, of the form
$\Ind_Z^N(K')$, where $K'/K_0$ is a
$Z$-Galois algebra, may be expressed as the
surjectivity of
\begin{equation}
\rH^1(Z_G(H),K'^\times)\to H^\vee
\label{eqn.groupH2}.
\end{equation}
Given this, the proof of
\cite[Prop.\ 2.2]{BnG} supplies an
equivalence of categories between
\begin{itemize}
\item 
$H$-Galois algebras over \'etale $K_0$-algebras and
\item 
$Z_G(H)$-Galois algebras with
equivariant homomorphism from $K'$;
\end{itemize}
in particular, there is then a $Z_G(H)$-Galois algebra $L/K_0$ with
homomorphism $K'\to L$,
compatible with the structure of Galois algebra for the
group $Z$, respectively $Z_G(H)$.
Assumption 1 is also implied by the existence
of such a Galois algebra $L$
and homomorphism $K'\to L$, as we see by using
the Hochschild-Serre spectral sequence and Hilbert's Theorem 90 to
identify $\rH^1(Z_G(H),K'^\times)$ with
$\rH^1(H\times Z_G(H),L^\times)$.
This allows us to view Assumption 1 as a lifting problem of Galois cohomology
\[ \rH^1(\Gal_{K_0},Z_G(H))\to \rH^1(\Gal_{K_0},Z) \]
and remark that the machinery of nonabelian cohomology
(cf.\ \cite[\S 1.3.2]{platonovrapinchuk})
supplies an obstruction to lifting in
$\rH^2(\Gal_{K_0},H)$.
\end{rema}

We now recall the definition of the {\em equivariant Burnside group}
$$
\Burn_n(G)=\Burn_{n,k}(G)
$$
following \cite[Section 4]{BnG}: it is a $\bZ$-module, generated by symbols
$$
\mathfrak s:=(H, N \actsfromleft K, \beta),
$$
where
\begin{itemize}
\item $H\subseteq G$ is an abelian subgroup,
\item $K\in \Alg_N(K_0)$, with $K_0\in \Bir_d(k)$, and $d\le n$,  
\item $\beta=(b_1,\dots,b_{n-d})$, a sequence of nonzero elements of the
character group $H^\vee$, that generate $H^\vee$.
\end{itemize}
The sequence of characters $\beta$ determines a faithful
$(n-d)$-dimensional representation of $H$ over $k$, with
trivial space of invariants.
As every $(n-d)$-dimensional representation of $H$ over $k$ splits as
a sum of one-dimensional representations, any
faithful $(n-d)$-dimensional representation of $H$ over $k$
determines a sequence of characters, generating $H^\vee$,
up to order.
The ambiguity of order gives us the first of several relations that we
impose on symbols:

\medskip
\noindent
\textbf{(O):}
$(H, N \actsfromleft K, \beta)=(H, N \actsfromleft K, \beta')$ if
$\beta'$ is a reordering of $\beta$.

\medskip
\noindent
The further relations are
\textbf{conjugation} and \textbf{blowup} relations:

\medskip
\noindent
\textbf{(C):} 
$(H,N \actsfromleft K,\beta) = (H',N' \actsfromleft K,\beta')$,
when $H'=gHg^{-1}$ and $N'=N_G(H')/H'$,  with $g\in G$, and $\beta$ and $\beta'$ are related by conjugation by $g$.

\medskip
\noindent
\textbf{(B1):}
$(H,N \actsfromleft K, \beta)=0$ when $b_1+b_2=0$.

\medskip
\noindent
\textbf{(B2):}
$(H,N \actsfromleft K, \beta)=\Theta_1+\Theta_2$,
where
\[
\Theta_1= \begin{cases} 0, & \text{if $b_1=b_2$}, \\ 
(H,N \actsfromleft K, \beta_1)+(H,N \actsfromleft K, \beta_2),
& \text{otherwise},
\end{cases}
\]
with
\begin{equation}
\label{eqn:beta12}
\beta_1:=(b_1,b_2-b_1,b_3,\ldots, b_{n-d}), \quad \beta_2:=(b_2,b_1-b_2,b_3,\ldots, b_{n-d}),
\end{equation}
and
\[
\Theta_2= \begin{cases} 0, & \text{if $b_i\in \langle b_1-b_2\rangle$ for some $i$}, \\
(\overline{H},\overline{N} \actsfromleft \overline{K}, \bar{\beta}), & \text{otherwise}, 
\end{cases}
\]
with
\[
\overline{H}^\vee:=H^\vee/\langle b_1-b_2\rangle,
\quad 
\bar{\beta}:=(\bar b_2,\bar b_3,\dots,\bar b_{n-d}), \quad
\bar b_i\in \overline{H}^\vee,
\]
and $\overline{K}$ carries the action described in Construction \textbf{(A)}
in \cite[Section 2]{BnG}, applied to the character $b_1-b_2$.

\medskip

We permit ourselves to write a symbol in the form
\begin{equation}
\label{eqnwithM}
(H,M\actsfromleft K,\beta)
\end{equation}
with a subgroup $M\subset N$ and $K\in \Alg_M(K_0)$,
with
\[ (H,M\actsfromleft K,\beta):=(H,N\actsfromleft \Ind_M^N(K),\beta). \]
We further allow $K_0$ to be a \emph{product} of fields; 
then \eqref{eqnwithM} will denote the corresponding
sum of symbols, one for each factor.

By Lemma~\ref{lem.Z}, any symbol in $\Burn_n(G)$ is of the form
$$
(H,Z\actsfromleft K,\beta),
$$
with $K\in \Alg_Z(K_0)$.
In this notation, Construction \textbf{(A)} has a compact formulation. 
Applied to a single character $b\in H^\vee$,
this yields the subgroup
\[ \overline{H}:=\ker(b) \subset H\]
and the symbol
\[ (\overline{H}, Z_G(H)/\overline{H}\actsfromleft K(t),\bar\beta), \]
where a $Z_G(H)$-action on $K(t)$ arises by lifting $b$
via \eqref{eqn.groupH2} and is trivial on $\overline{H}$,
and $\bar\beta$ is obtained from
$\beta$ by applying the map $H^\vee\to \overline{H}^\vee$, as above.

\begin{rema}
\label{rema:iter}
Construction \textbf{(A)} may be applied to a collection of
characters, yielding the same outcome as when applied
iteratively, one character at a time.
\end{rema}

\medskip

A $G$-action on $X$ in standard form always satisfies

\medskip

{\bf Assumption 2}:  The stabilizers for the $G$-action on $X$ are abelian,
and for every $H$ and $F$ in \eqref{eqn:class} the
composite homomorphism
\[ \Pic^G(X)\to \rH^1(N_G(H),k(F)^\times)\to H^\vee \]
is surjective,
where the first map is given by restriction and the second is the
map from Assumption 1, with $K=k(F)$.

\

Note that Assumption 2 implies Assumption 1,
for every $H$ and every $N_G(H)/H\actsfromleft k(F)$
(see \cite[Rmk.\ 3.2(i)]{BnG}).

A variant, that will occur below, is the requirement of surjectivity, when we
restrict to a given subgroup of $\Pic^G(X)$.
Given this, we will say that Assumption 2 holds for the given subgroup
of $\Pic^G(X)$.

\section{Filtrations}
\label{sec.filtrations}
In this section, we explore additional combinatorial constructions on equivariant Burnside groups $\Burn_n(G)$, reflecting the geometry of the $G$-action on strata with given generic stabilizers.

\begin{defi}
\label{defn:pre}
A $G$-prefilter 
is a collection $\mathbf{H}$ of pairs $(H,Y)$
consisting of an abelian subgroup
$H\subseteq G$ and a subgroup
$$
Y\subseteq Z=Z_G(H)/H,
$$ 
such that $\mathbf{H}$ is closed under conjugation,
i.e., for $(H,Y)\in \mathbf{H}$ we have
$$
(gHg^{-1},gYg^{-1})\in \mathbf{H},\quad \text{ for all } \, g\in G.
$$
\end{defi}

\begin{defi}
\label{defn.BurnH}
Given a $G$-prefilter $\mathbf{H}$, we let
$$
\Burn^{\mathbf{H}}_n(G)
$$ 
be the quotient of
$\Burn_n(G)$ by the subgroup generated by classes of the form
$$
(H, Y \actsfromleft K, \beta),
$$
where $K\in \Alg_Y(K_0)$ is a \emph{field}, and
\[ (H,Y)\notin \mathbf{H}. \]
\end{defi}

\begin{prop}
\label{prop.filtration}
Let $\mathbf{H}$  be a $G$-prefilter 
such that if $(H,Y)\in \mathbf{H}$, with $H$ nontrivial,
then $(\langle H,g\rangle,Y/\langle \bar g\rangle)\in \mathbf{H}$ for all
$g\in Z_G(H)$ satisfying
$$
\bar g\in Y\qquad\text{and}\qquad Y\subseteq Z_G(g)/H.
$$
Then $\Burn^{\mathbf{H}}_n(G)$ is generated by
triples
\[ (H, Y \actsfromleft K, \beta), \]
where $K\in \Alg_Y(K_0)$ is a field and
\[ (H,Y)\in \mathbf{H}, \]
subject to relations \emph{\textbf{(O)}}, \emph{\textbf{(C)}}, \emph{\textbf{(B1)}}, and \emph{\textbf{(B2)}} 
applied to these triples. 
\end{prop}

\begin{proof}
For any
$$
(H, Y \actsfromleft K, \beta)
$$ 
with
$K\in \Alg_Y(K_0)$ a field,
the term $\Theta_2$ from \textbf{(B2)}, when nontrivial, consists of a subgroup
$\ker(b)$ of $H$ for some $b\in H^\vee$,
a field $K(t)$ with action of the
pre-image of $Y$ in $Z_G(H)/\ker(b)$,
and a sequence of characters.
If $(H,Y)\notin \mathbf{H}$, then by hypothesis the pair consisting of
$\ker(b)$ and the pre-image of $Y$
is not in 
$\mathbf{H}$.
This observation establishes the proposition,
since \textbf{(B1)} involves just one triple,
and in $\Theta_1$ from \textbf{(B2)} the group and algebra do not change.
\end{proof}

\begin{exam}
\label{exa.filtration}

\noindent
\begin{itemize}
\item For $G$ abelian, we have
$$
\Burn^G_n(G)=\Burn^{\{(G,\mathrm{triv})\}}_n(G),
$$
where the left side was introduced in \cite[\S 8]{BnG}:
this is the quotient of
$\Burn_n(G)$ by all triples whose first entry is a proper subgroup of $G$.
\item 
For $\mathbf{H}$ consisting of all $(H,Y)$ with $H$ nontrivial cyclic and
$Y$ noncyclic, and $k$ algebraically closed, 
$\Burn^{\mathbf{H}}_2(G)$ appeared in 
\cite[\S 7.4]{HKTsmall}. 
\end{itemize}
\end{exam}

\begin{rema}
\label{rema:alt}
One can additionally suppress the field information, which will lead to
{\em combinatorial} analogues of Burnside groups. We will explore this in Section~\ref{sec.combinatorial}.
\end{rema}

\section{Nontrivial generic stabilizers}
\label{sec.indexed}
In this section, we introduce a version of the equivariant Burnside group, relevant for considerations of actions with {\em nontrivial} generic stabilizer. 

Let $G$ be a finite group.
A variant of the equivariant Burnside group takes
the additional data of a finite index set
\[ I\subset \N. \]
The {\em equivariant indexed Burnside group}
$$
\Burn_{n,I}(G),
$$
is defined as a quotient of the $\bZ$-module generated by symbols
$$
(H\subseteq H', N' \actsfromleft K, \beta, \gamma),
$$
where
\begin{itemize}
\item $H\subseteq H'$ are abelian subgroups of $G$, 
\item $N':=N_{N_G(H)}(H')/H'$,
\item $K\in \Alg_{N'}(K_0)$, with $K_0\in \Bir_d(k)$, and $d\le n-|I|$,
\item $\beta=(b_1,\dots,b_{n-d-|I|})$, a sequence of nonzero
characters of $H'$,
trivial upon restriction to $H$,
that generate $(H'/H)^\vee$,
\item $\gamma=(c_i)_{i\in I}$ is a sequence of
elements of $H'^\vee$, such that
the images of $c_i$
in $H^\vee$ generate $H^\vee$.
\end{itemize}
As in Section \ref{sect:gen}, we permit ourselves to write a symbol in the form
$$
(H\subseteq H', M' \actsfromleft K, \beta, \gamma),
$$
where $M'\subset N'$ is a subgroup.
Every symbol may be expressed as
\[ (H\subseteq H',Z'\actsfromleft K,\beta,\gamma),\qquad
Z':=Z_G(H')/H'. \]
(Notice that $Z_G(H')=Z_{N_G(H)}(H')$.)

These symbols are subject to relations: 

\medskip
\noindent
\textbf{(O):}
$(H\subseteq H', N' \actsfromleft K, \beta,\gamma)=(H\subseteq H', N' \actsfromleft K, \beta',\gamma)$ if
$\beta'$ is a reordering of $\beta$. 

\medskip
\noindent
\textbf{(C):} 
$(H\subseteq H',N' \actsfromleft K,\beta,\gamma) = (gHg^{-1} \subseteq gH'g^{-1},gN'g^{-1} \actsfromleft K,\beta',\gamma')$
for $g\in G$, with $\beta$ and $\beta'$,
respectively $\gamma$ and $\gamma'$, related by conjugation by $g$.

\medskip
\noindent
\textbf{(B1):}
$(H\subseteq H',N' \actsfromleft K, \beta, \gamma)=0$ when $b_1+b_2=0$.

\medskip
\noindent
\textbf{(B2):}
$(H\subseteq H',N' \actsfromleft K, \beta, \gamma)=\Theta_1+\Theta_2$,
where $\Theta_1$ and $\Theta_2$ are as in Section \ref{sect:gen}, with $H$ prepended and
$\gamma$, respectively $\bar\gamma$,
appended to the corresponding symbols.

\begin{rema}
\label{rem.charactersgenerate}
By analogy with Remark \eqref{rem.groupH2},
we may express Assumption 1, for the
Galois algebra $K$, as the surjectivity of the
middle vertical map
\[
\xymatrix@C=8pt{
0 \ar[r] &
\rH^1(Z_G(H')/H,K^\times) \ar[r] \ar[d] &
\rH^1(Z_G(H'),K^\times) \ar[r] \ar[d] &
\rH^1(H,K^\times)^{Z_G(H')/H} \ar[d] \\
0 \ar[r] & (H'/H)^\vee \ar[r] & H'^\vee \ar[r] & H^\vee \ar[r] & 0
}
\]
Here, the top row comes from the Hochschild-Serre spectral sequence.
In a symbol, we have $\beta$ generating the left-hand group in the bottom row,
while $\gamma$ is a sequence of
characters of $H'$, whose images generate $H^\vee$.
Consequently, $\beta$ and $\gamma$
together generate $H'^\vee$.
Thus we have a homomorphism
\[
\psi_I\colon \Burn_{n,I}(G)\to \Burn_n(G),
\]
sending $(H\subseteq H',Z'\actsfromleft K,\beta,\gamma)$ to
\[
(H',Z'\actsfromleft K, \beta\cup\gamma)
\]
when $\gamma$ is a sequence of nontrivial characters,
otherwise to $0$.
\end{rema}

In order to explain the
relevance of this definition, we introduce a map which
converts some of the characters in $\gamma$ to a
transcendental extension of the Galois algebra.
Let
\[ J\subseteq I \]
be a subset.
Given a symbol $(H\subseteq H',Z'\actsfromleft K,\beta,\gamma)$,
we use $J$ to define subgroups
\begin{align*}
\overline{H}'&:=\bigcap_{i\in I\setminus J}\ker(c_i)\subseteq H', \\
\overline{H}&:=H\cap \overline{H}'\subseteq H.
\end{align*}
Then we define
\[ \omega_{I,J}\colon \Burn_{n,I}(G)\to \Burn_{n,J}(G), \]
by applying Construction \textbf{(A)} when possible:
\[
(H\subseteq H',Z'\actsfromleft K,\beta,\gamma)\mapsto
(\overline{H}\subseteq \overline{H}',
Z_G(H')/\overline{H}'\actsfromleft K((t_i)_{i\in I\setminus J}),\bar\beta,
\bar\gamma), 
\]
where $\bar\gamma=(\bar c_j)_{j\in J}$, when all of the characters of $\bar\beta$ are nonzero, and
$$
(H\subseteq H',Z'\actsfromleft K,\beta,\gamma)\mapsto 0, \quad \text{ otherwise.}
$$
This is compatible with relations:
the only one that is nontrivial to check is \textbf{(B2)},
where
$\Theta_1$ maps to
$\overline{\Theta}_1$, as we see by dividing into cases according to the vanishing of $\bar b_1$ or $\bar b_2$, or their equality, and
$\Theta_2$ maps to $\overline{\Theta}_2$, as we see using Remark~\ref{rema:iter}.

We recall the setting of
\cite[Defn.\ 5.4]{BnG}: Let 
$X$ be a smooth projective variety of dimension $n$, with a generically free action of $G$, satisfying Assumption 2. Let $D_1$, $\dots$, $D_{\ell}$ be $G$-stable divisors, 
with
$$
D_I:=\bigcap_{i\in I}D_i, \quad \text{
for } I\subseteq \cI:=\{1,\dots,\ell\}, \quad D_\emptyset=X.
$$
We suppose, for notational simplicity,
that for every $I$ the generic stabilizers of the
components of $D_I$
belong to a single conjugacy class of subgroups,
and take $H_I$ to be a representative.
Then
to $I\subseteq M\subseteq \cI$
we attach the following class in
$\Burn_{n,I}(G)$:
\[
\chi_{I,M}(X\actsfromright G,(D_i)_{i\in \cI}):=\sum_{H'\supseteq H_I}\sum_{\substack{W\subset D_I\\
\mathrm{generic\ stabilizer}\ H'\\
\{i\in \cI\,|\,W\subset D_i\}=M}}
\!\!\!(H_I\subseteq H',N'\actsfromleft k(W),\beta,\gamma),
\]
where 
\begin{itemize}
\item the first sum is over conjugacy class representatives $H'$ of abelian subgroups
of $N_G(H_I)$, containing $H_I$,
\item the second sum is over $N_{N_G(H_I)}(H')$-orbits of components
$W$ with generic stabilizer $H'$, contained in components of
$D_I$ with generic stabilizer $H_I$ and satisfying
$\{i\in \cI\,|\,W\subset D_i\}=M$,
\item $\beta=\beta_W(D_I)$ encodes the normal bundle to
$W$ in $D_I$, and
\item 
$\gamma=(c_i)_{i\in I}$, the characters coming from
$D_i$ with $i\in I$.
\end{itemize}

Then
\[
[\mathcal{N}_{D_I/X}\actsfromright G]^{\mathrm{naive}}=
\sum_{I\subseteq M\subseteq\cI}
\sum_{M\setminus I\subseteq J\subseteq M}
\!\psi_{I\cap J}(\omega_{I,I\cap J}(\chi_{I,M}(X\actsfromright G,(D_i)_{i\in \cI}))),
\]
where the terms with $J=\emptyset$ contribute
$[\mathcal{N}^{\circ}_{D_I/X}\actsfromright G]^{\mathrm{naive}}$.
This provides some insight to \cite[Lemma 5.7]{BnG}.

\section{Fibrations}
\label{sec.fibrations}
In this section, we define a projectivized version
of the equivariant indexed Burnside group and use it to give a formula for the class in $\Burn_n(G)$ of the projectivization of a sum of line bundles.   

Let $G$ be a finite group and $I\subset \N$ a \emph{nonempty} finite index set.
The \emph{equivariant projectively indexed Burnside group}
\[ \Burn_{n,\bP(I)}(G) \]
is defined with generators and relations as in
Section \ref{sec.indexed}, where
\begin{itemize}
\item $\beta$ consists of $n-d-|I|+1$ characters
(so $d\le n-|I|+1$),
\item 
the \emph{differences} of pairs of characters of 
$\gamma$ should generate $H^\vee$,
\item 
and there is an additional relation:
\end{itemize}

\medskip
\noindent
\textbf{(P):}
If
$\gamma'-\gamma$ is a constant sequence then
$$
(H\subseteq H', N' \actsfromleft K, \beta,\gamma)=(H\subseteq H', N' \actsfromleft K, \beta,\gamma').
$$

\

We define
\[ \omega_{\bP(I),J}\colon \Burn_{n,\bP(I)}(G)\to \Burn_{n,J}(G), \]
for a \emph{proper} subset
\[ J\subsetneq I, \]
by 
\begin{itemize}
\item 
choosing $i_0\in I\setminus J$,
\item 
applying \textbf{(P)} to get a representative symbol
$$
(H\subseteq H', N' \actsfromleft K, \beta,\gamma)
$$ 
with
$\gamma_{i_0}=0$, and
\item 
applying
$\omega_{I\setminus\{i_0\},J}$ to the class of
$(H\subseteq H', N' \actsfromleft K, \beta,(c_i)_{i\in I\setminus \{i_0\}})$ in $\Burn_{n,I\setminus \{i_0\}}(G)$.
\end{itemize}

Let $X$ be a smooth projective variety over $k$.
Assume that $X$ carries a $G$-action, and let
$L_0$, $\dots$, $L_r$ be $G$-linearized line bundles on $X$.
The next statement examines the condition, for
$G$ to act generically freely on $\bP(L_0\oplus\dots\oplus L_r)$, so
that Assumption 2 satisfied.

\begin{lemm}
\label{lem.assumptionbundles}
Let $X$ be a smooth projective variety over $k$ with a
$G$-action
and $G$-linearized line bundles $L_0$, $\dots$, $L_r$.
Let $H$ be the stabilizer at the generic point of a component of $X$,
and let us denote the $N_G(H)$-orbit of the component by $X'$.
The following are equivalent.
\begin{itemize}
\item[(i)] The $N$-action on $X'$ satisfies Assumption 2, and
$H$ is abelian with $H^\vee$ spanned by the differences of
characters determined by $L_0$, $\dots$, $L_r$.
\item[(ii)] The $G$-action on $\bP(L_0\oplus\dots\oplus L_r)$ is
generically free and satisfies Assumption 2.
\item[(iii)] The $G$-action on $\bP(L_0\oplus\dots\oplus L_r)$ is
generically free and satisfies Assumption 2 for
$L_0$, $\dots$, $L_r$, together with the
$G$-linearized line bundles on $X$ associated with
$N$-linearized line bundles on $X'$.
\end{itemize}
\end{lemm}

The statement is inspired by \cite[Lemma 7.3]{BnG}.

\begin{proof}
The action of $G$ on $\bP(L_0\oplus\dots\oplus L_r)$ is generically free
if and only if the
action of $N_G(H)$ on $\bP(L_0|_{X'}\oplus\dots\oplus L_r|_{X'})$ is generically free.
The latter has generic stabilizer $\bigcap_{i=1}^r\ker(b_i-b_0)$.
Thus the condition on $H$ in (i) is equivalent to the condition of
generically free action in (ii) and in (iii).
We assume this from now on.

An $N$-linearized line bundle on $X'$ determines an
$N_G(H)$-linearized line bundle on $X'$, with trivial $H$-action.
An $N_G(H)$-linearized line bundle on $X'$ determines, and is determined by,
$G$-linearized line bundle on $X$; this is the meaning of the
associated line bundles in (iii).
Conversely, a $G$-linearized line bundle on $X$ which restricts to an
$N_G(H)$-linearized line bundle on $X'$ with trivial $H$-action,
gives rise to an $N$-linearized line bundle on $X'$.

We start by showing (i) implies (iii), using the
interpretation of Assumption 2 in terms of the
representability of a certain morphism from the quotient stack to a
product of copies of $B\G_m$.
Given (i), we have such a representable morphism
\[ [X'/N]\to B\G_m\times\dots\times B\G_m. \]
Correspondingly, the fibers of the composite morphism
\[ [X/G]\to [X'/N]\to B\G_m\times\dots\times B\G_m \]
all have constant stabilizer group $H$.
The condition in (i) implies that the
$H$-representation given by $b_0$, $\dots$, $b_r$ is faithful.
With $r+1$ additional factors $B\G_m$ we get a
representable morphism from $[X/G]$, hence also from
$\bP(L_0\oplus\dots\oplus L_r)$.

Since trivially (iii) implies (ii), it remains only
to show (ii) implies (i).
Generally, a line bundle on a projective bundle is isomorphic to
the pullback of
a line bundle from the base, twisted by a power of the tautological line bundle.
Two vector bundles, one obtained from the other by tensoring
by a line bundle, have isomorphic projectivizations,
the tautological line bundle of one obtained from the other by
tensoring by the pullback of the line bundle from the base.
A sum of line bundles, after such tensoring, may be brought in a form with
trivial $i$th factor, for any $i$,
and this way see that any power of the tautological line bundle on
$\bP(L_0\oplus\dots\oplus L_r)$, restricted to
the open $U_i\subset \bP(L_0\oplus\dots\oplus L_r)$
defined by nonvanishing on the component $L_i$,
is identified with a line bundle pulled back from the base;
all of these assertions are valid in an equivariant setting.
With the notation of Assumption $2$ for $\bP(L_0\oplus\dots\oplus L_r)$,
we always have $\Spec(k(F))\subset U_i$ for some $i$.
So, (ii) implies that
the $G$-action on $\bP(L_0\oplus\dots\oplus L_r)$
satisfies Assumption 2 for $\Pic^G(X)$.
Since
\[ \bP(L_0\oplus\dots\oplus L_r)\to X \]
admits
equivariant sections, we deduce that some finite collection of
$G$-linearized line bundles on $X$ determines a representable morphism
\[ [X/G]\to B\G_m\times\dots\times B\G_m. \]
Replacing each by a tensor product with combinations of
$L_0$, $\dots$, $L_r$, we obtain a
$G$-linearized line bundle on $X$ that comes from an
$N$-linearized line bundle on $X'$.
The corresponding morphism
\[ [X'/N]\to B\G_m\times\dots\times B\G_m \]
is representable, and thus we have (i).
\end{proof}

\begin{prop}
\label{prop.fibration}
Let $X$ be a smooth projective variety of dimension $n-r$ over $k$ with a
$G$-action and $G$-linearized line bundles $L_0$, $\dots$, $L_r$.
We assume the conditions and adopt the notation of
Lemma \ref{lem.assumptionbundles}.
We define $I:=\{0,\dots,r\}$ and the following class in $\Burn_{n,\bP(I)}(G)$:
\[
\xi(X\actsfromright G,(L_i)_{i\in I}):=
\sum_{H'\supseteq H} \sum_{\substack{W\subset X'\\ \mathrm{generic\ stabilizer}\ H'}}
(H\subseteq H',N'\actsfromleft k(W),\beta,\gamma),
\]
where
\begin{itemize}
\item the first sum is over
abelian subgroups $H'$ of $G$ that contain $H$,
up to conjugacy in $N_G(H)$,
\item the second sum is over $N_{N_G(H)}(H')$-orbits of components $W\subset X'$ where the
generic stabilizer is $H'$,
\item $\beta=\beta_W(X')$ encodes the normal bundle to $W$ in $X'$, and
\item $\gamma=(c_i)_{i\in I}$, the characters coming from $L_i$ with $i\in I$.
\end{itemize}
Then
\[
[\bP(L_0\oplus\dots\oplus L_r) \actsfromright G]=\sum_{J\subsetneq I} \psi_J(\omega_{\bP(I),J}(\xi(X\actsfromright G, (L_i)_{i\in I})))
\]
in $\Burn_n(G)$.
\end{prop}

\begin{proof}
We identify each contribution to
$[\bP(L_0\oplus\dots\oplus L_r) \actsfromright G]$
as 
$$
V=\varphi_J^{-1}(W),
$$
for some $W$ in the
definition of $\xi(X\actsfromright G,(L_i)_{i\in I})$,
where $\varphi_J$ denotes the projection to $X$ from the projectivization of
$\bigoplus_{i\in I\setminus J}L_i$.
Then,
$$
(H\subseteq H',N'\actsfromleft k(W),\beta,\gamma)\in \Burn_{n,\bP(I)}(G)
$$
maps under $\psi_J\circ \omega_{\bP(I),J}$ to
$$
(\overline{H}',N_{N_G(H)}(\overline{H}')\actsfromleft k(V),\beta_V(X)).
$$
\end{proof}

\begin{exam}
\label{exa.Z5S3}
Let $G:=C_5\times \fS_3$, acting on
$X:=\bP^1$ via an irreducible $2$-dimensional representation of $\fS_3$.
We take $L_0$ to be trivial and
$L_1$ to be the twist of
$\cO_{\bP^1}(1)$ by a
nontrivial character $\chi$ of $C_5$.
Then we have the situation of
Lemma \ref{lem.assumptionbundles} with
$H=C_5$ and $N=\fS_3$, and the conditions
of the lemma are satisfied.
We have
\begin{align*}
\xi(X\actsfromright G,(L_0,L_1))&=
(C_5\subseteq C_5,\fS_3\actsfromleft k(\bP^1),\emptyset,(0,\chi))\\
&+
(C_5\subseteq C_5\times \langle (1,2)\rangle ,\mathrm{triv}\actsfromleft k,(0,1),(0,(\chi,0)))\\
&+
(C_5\subseteq C_5\times \langle (1,2)\rangle,\mathrm{triv}\actsfromleft k,(0,1),(0,(\chi,1)))\\
&+
(C_5\subseteq C_5\times \fA_3,\fS_3/\fA_3\actsfromleft k\times k,(0,1),(0,(\chi,1))).
\end{align*}
The outcome of Proposition \ref{prop.fibration} is
\begin{align*}
[&\bP(L_0\oplus L_1)\actsfromright G]=
(\mathrm{triv},G\actsfromleft k(\bP^1)(t),\emptyset)+
(\langle (1,2)\rangle,C_5\stackrel{\chi}{\actsfromleft} k(t),1)\\
&+{\color{red} (C_5,\fS_3\actsfromleft k(\bP^1),\chi)}+
(C_5\times\langle (1,2)\rangle,\mathrm{triv}\actsfromleft k,((0,1),(\chi,0)))\\
&+(C_5\times\langle (1,2)\rangle,\mathrm{triv}\actsfromleft k,((0,1),(\chi,1)))\\
&+(C_5\times \fA_3,\fS_3/\fA_3\actsfromleft k\times k,((0,1),(\chi,1)))\\
&+{\color{red} (C_5,\fS_3\actsfromleft k(\bP^1),-\chi)}+
(C_5\times\langle (1,2)\rangle,\mathrm{triv}\actsfromleft k,((0,1),(-\chi,0)))\\
&+(C_5\times\langle (1,2)\rangle,\mathrm{triv}\actsfromleft k,((0,1),(-\chi,1)))\\
&+(C_5\times \fA_3,\fS_3/\fA_3\actsfromleft k\times k,((0,1),(-\chi,1))).
\end{align*}
In the notation of Section~\ref{sec.filtrations} we 
observe that the $G$-prefilter
\[
\mathbf{H}:=\{(C_5,\fS_3)\}
\]
satisfies the condition of Proposition \ref{prop.filtration}.
Upon projection 
$$
\Burn_2(G)\to \Burn^{\mathbf{H}}_2(G)
$$
(see Definition~\ref{defn.BurnH}), we obtain the class  
$$
{\color{red}
(C_5,\fS_3\actsfromleft k(\bP^1),\chi)
} + 
{\color{red}
(C_5,\fS_3\actsfromleft k(\bP^1),-\chi)}
\in 
\Burn^{\mathbf{H}}_2(G).
$$
This class is nonzero. Moreover, it is different for $\chi\in \{\pm 1\}$ as compared to 
$\chi\in \{\pm 2\}$.

Geometrically, the situation above arises as follows: Consider the 3-dimensional 
representation $W_\chi = 1\oplus (V\otimes \chi)$ of $G$, sum of a trivial $1$-dimensional representation and twist by $\chi$ of the
standard $2$-dimensional representation
$V$ of $\fS_3$. This gives a generically free action of $G$ on $\bP^2=\bP(W_{\chi})$, 
with a $G$-fixed point $\mathfrak{p}$. 
To bring the $G$-action into a form where
Assumption 2 is satisfied, we need to blow up $\mathfrak{p}$, and
$$
[\bP(W_{\chi}) \actsfromright G] = [\bP(L_0\oplus L_1)\actsfromright G] \in \Burn_2(G).
$$
\end{exam}

\section{Products}
\label{sect:prod}

Let $G'$ and $G''$ be finite groups. 
Define a product map
\[ 
\Burn_{n'}(G')\times \Burn_{n''}(G'')\to \Burn_{n'+n''}(G'\times G''). 
\]
On symbols, it is given by
\begin{equation}\label{eqn:prod}
((H', Z'\actsfromleft K',\beta'), (H'', Z''\actsfromleft K'',\beta''))\mapsto
(H, Z\actsfromleft K,\beta), 
\end{equation}
where
\begin{itemize}
\item $H =H'\times H''$,
\item $Z=Z'\times Z''$,
\item $K=K'\otimes_{k} K''$, with the natural action of $Z$,
\item $\beta=\beta'\cup \beta''$.
\end{itemize}

\begin{prop}
The product map \eqref{eqn:prod} is well-defined, and satisfies
$$
([X'\actsfromright G'],[X''\actsfromright G''])\mapsto [X'\times X''\actsfromright G'\times G''].
$$
\end{prop}

\begin{proof}
The map clearly respects relations.
The only point to remark is that in
\textbf{(B2)}, the condition for
nontriviality of $\Theta_2$ holds for
$\beta'$ if and only if it holds for
$\beta=\beta'\cup \beta''$.
\end{proof}

\section{Restrictions}
\label{sec.restr}
Let $G$ be a finite group and $G'\subset G$ a subgroup. 
A $G$-action on a quasiprojective variety $X$ induces an action of $G'$,
and thus it is natural to propose the existence of a restriction homomorphism
from $\Burn_n(G)$ to $\Burn_n(G')$, acting by
\begin{equation}
\label{eqn:restr}
[X\actsfromright G]\mapsto [X\actsfromright G'].
\end{equation}
In this section we establish the existence and uniqueness of this homomorphism. 

\begin{exam}
Suppose $H$ is an abelian subgroup of $G$, contained in $G'$.
Symbols, identified in $\Burn_n(G)$ by relation \textbf{(C)},
might no longer be identified in $\Burn_n(G')$.
E.g., with $G=\mathfrak{D}_4$ and $G'=C_4$ the restriction of
$(G',G/G'\actsfromleft k\times k,1)\in \Burn_1(G)$
to $\Burn_1(G')$ has to be a sum of two symbols with distinct characters:
\[
(G',G/G'\actsfromleft k\times k,1)
\mapsto
(G',\mathrm{triv}\actsfromleft k,1)+(G',\mathrm{triv}\actsfromleft k,3).
\]
\end{exam}

\begin{theo}
\label{theo:res}
For all $n\ge 0$, 
there exists a unique homomorphism of abelian groups
$$
\mathrm{res}^G_{G'}: \Burn_n(G)\to \Burn_n(G').
$$
compatible with \eqref{eqn:restr}. 
\end{theo}

\begin{proof}
By Lemma \ref{lem.Z}, it suffices to consider
symbols of the form
$$
\mathfrak{s}=(H,Z\actsfromleft K,\beta).
$$ 
When we act by conjugation by some element of $G$, we obtain an
equivalent symbol, where $H$ is replaced by a conjugate,
the corresponding centralizer quotient replaces $Z$, and
conjugation is used to form from $\beta$ a sequence of
characters of the conjugate of $H$.
By conjugation we have a transitive action of $G$ on a set $\mathfrak{S}$ of
symbols, where $\mathfrak{s}\in \mathfrak{S}$ has stabilizer $Z_G(H)$.
The restriction of the action to $G'$ consists of finitely many orbits;
in the formula below the sum is over
orbit representatives
\[ \mathfrak{s}'=(H',Z'\actsfromleft K,\beta'), \]
such that the restriction $\beta'|_{H'\cap G'}$ of $\beta'$ to $H'\cap G'$ has trivial space of
invariants; here, $Z'$ denotes $Z_G(H')/H'$.
Then we define the restriction to $G'$ by
\[ \mathfrak{s} \mapsto \sum_{\mathfrak{s}'}
(H'\cap G',(Z_G(H')\cap G')/(H'\cap G')\actsfromleft K,
\beta'|_{H'\cap G'}); \]
this map respects relations.
Uniqueness follows from \cite[Rmk.\ 5.16]{BnG}.
\end{proof}

As an application of the restriction construction, we obtain a map 
$$
\Burn_{n'}(G)\times \Burn_{n''}(G)\to \Burn_{n'+n''}(G),
$$
using the product construction in 
Section~\ref{sect:prod} with $G'=G''=G$, followed restriction to the diagonal
$$
G\subset G\times G. 
$$
This map on Burnside groups sends
$$
([X'\actsfromright G],[X''\actsfromright G])\mapsto [X'\times X''\actsfromright G].
$$

\section{Combinatorial analogs}
\label{sec.combinatorial}

Here we define and study a quotient
$$
\Burn_n(G)\ra \BC_n(G)
$$
to a {\em combinatorial} version of the equivariant Burnside group, by forgetting the 
information about the Galois algebra. 

\begin{defi}
\label{defn:comb}
The \emph{combinatorial symbols group} 
$$
\BC_n(G)
$$
is the $\bZ$-module,
generated by symbols
\[ (H,Y,\beta) \]
with
$H$ abelian,
$Y\subseteq Z_G(H)/H$,
and $\beta$ a sequence of
nonzero elements generating $H^\vee$,
of length at most $n$,
modulo relations:

\medskip
\noindent
\textbf{(O):}
$(H,Y,\beta)=(H,Y,\beta')$ if
$\beta'$ is a reordering of $\beta$.

\medskip
\noindent
\textbf{(C):} 
$(H,Y,\beta) = (gHg^{-1},gYg^{-1},\beta')$
for $g\in G$, with $\beta$ and $\beta'$
related by conjugation by $g$.

\medskip
\noindent
\textbf{(B1):}
$(H,Y, \beta)=0$ when $b_1+b_2=0$.

\medskip
\noindent
\textbf{(B2):}
$(H,Y, \beta)=\Theta_1+\Theta_2$,
where $\Theta_1$ and $\Theta_2$ are as in Section \ref{sect:gen}, 
i.e., 
\[
\Theta_1= \begin{cases} 0, & \text{if $b_1=b_2$}, \\ 
(H,Y,\beta_1)+(H,Y, \beta_2),
& \text{otherwise},
\end{cases}
\]
with $\beta_1$ and $\beta_2$ as in \eqref{eqn:beta12},
and
\[ \Theta_2= \begin{cases} 0, & \text{if $b_i\in \langle b_1-b_2\rangle$ for some $i$}, \\
(\overline{H},\overline{Y}, \bar{\beta}), & \text{otherwise}, 
\end{cases}
\]
where $\overline{H}=\ker(b_1-b_2)$,
$\overline{Y}$ is the pre-image of $Y$ in
$Z_G(H)/\overline{H}$, and $\bar \beta$ consists
of the restrictions to $\overline{H}$ of the characters of $\beta$.
\end{defi}

\begin{prop}
\label{prop:comb}
The map sending the class of a triple
$$
(H, Y\actsfromleft K, \beta) \in \Burn_n(G),
$$
for fields $K\in \Alg_Y(K_0)$, $K_0\in \Bir_d(k)$, with $d\le n$, 
to 
$$
[k':k](H,Y,\beta) \in \BC_n(G), 
$$
where $k'$ is the algebraic closure of $k$ in $K_0$,
gives a surjective homomorphism 
$$
\Burn_n(G) \ra  \BC_n(G).
$$
\end{prop}

\begin{proof}
This is clear from the description of the relations in
$\Burn_n(G)$ from Section \ref{sect:gen}.
\end{proof}

\begin{defi}
\label{def.BCH}
Given a $G$-prefilter $\mathbf{H}$, we let
\[
\BC_n^{\mathbf{H}}(G)
\]
be the quotient of $\BC_n(G)$ by the subgroup
generated by classes $(H,Y,\beta)$ with
$(H,Y)\notin \mathbf{H}$.
\end{defi}

Exactly as in Section \ref{sec.filtrations} we have

\begin{prop}
\label{prop.BCH}
Let $\mathbf{H}$ be a $G$-prefilter, satisfying the
hypothesis of Proposition \ref{prop.filtration}.
Then $\BC_n^{\mathbf{H}}(G)$ is generated by symbols
$(H,Y,\beta)$ for $(H,Y)\in \mathbf{H}$,
subject to relations \emph{\textbf{(O)}}, \emph{\textbf{(C)}}, \emph{\textbf{(B1)}}, and \emph{\textbf{(B2)}} 
applied to these symbols. 
\end{prop}

Additionally, upon passage to the combinatorial analogue we also have the other structures developed in this paper:
\begin{itemize}
\item equivariant (projectively) indexed combinatorial Burnside group;
\item product map;
\item restriction homomorphisms.
\end{itemize}

\begin{exam}
\label{exa.BCH}
Suppose that $G$ is abelian.
\begin{itemize}
\item We have (cf.\ \cite[\S 8]{BnG})
\[ \cB_n(G)=\BC_n^{(G,\mathrm{triv})}(G), \]
where $\cB_n(G)$ is the symbols group from \cite{kontsevichpestuntschinkel}.
\item There is a commutative diagram
\[
\xymatrix{
\Burn_n(G) \ar[d] \ar[r]  &   \BC_n(G) \ar[d] \\
   \Burn^G_n(G) \ar[r]    &    \cB_n(G)
}
\]
(The factor factor $[k':k]$ in Proposition \ref{prop:comb} matches the
similar factor in \cite[Prop.\ 8.1]{BnG}.)
\end{itemize}
\end{exam}

\section{Applications}
\label{sect:appl}

As a first application of the formalism in \cite{BnG} for nonabelian groups, we gave
in \cite{HKTsmall} an example of $G=C_2\times \fS_3$-actions on $\bP^2$ and a quadric surface $Q\subset \bP^3$,
which were distinguished by the respective classes in $\Burn_2(G)$. 
The actions are {\em stably} $G$-equivariantly rational \cite{lemire}.

Here we give a further application, for $G=\fS_4$, acting on $\bP^2$ and a del Pezzo surface of degree 6. This example
was treated in  \cite{BanTok}, via birational rigidity techniques (Noether inequality).

\

\noindent
We recall basic facts about the subgroup lattice of $G=\fS_4$:
\begin{itemize}
\item Conjugacy classes of  nonabelian subgroups:  $\fS_3$, $\sD_4$, $\fA_4$, 
\item Conjugacy classes of abelian subgroups: 
trivial, even $\bZ/2$, odd $\bZ/2$, $\bZ/3$, $\bZ/4$, even $\mathfrak K_4\simeq \bZ/2\oplus \bZ/2$, odd $\mathfrak K_4$.
\end{itemize}

\

\noindent
\textbf{First action:} On $\bP^2$, we consider the projectivization of the standard 3-dimensional representation  $V_3$,
with respect to basis 
$$
(-1,1,1,-1), \quad (1,-1,1,-1), \quad (1,1,-1,-1),
$$
given by the 4 matrices below. 
$$
\sigma:=\begin{pmatrix} 
0 & 1 & 0 \\ 
1 & 0 & 0\\
0 & 0 & 1
\end{pmatrix}
\quad 
\tau:=
\begin{pmatrix}
0 & 0 & 1\\
1 & 0 & 0\\
0 & 1 & 0
\end{pmatrix} 
$$
$$
\lambda_1:=
\begin{pmatrix}
-1& 0& 0\\
0 & 1 & 0\\
 0& 0 & -1 
\end{pmatrix} 
\lambda_2:=
\begin{pmatrix}
-1& 0 & 0\\
0& -1& 0\\
0 & 0 & 1
\end{pmatrix} 
$$
Here $\sigma$ and $\tau$ generate $\fS_3$ and $\lambda_1,\lambda_2$ the even
$\mathfrak K_4$.

The restriction of $V_3$ to $\sD_4$ decomposes as 
$1$-dimensional plus irreducible $2$-dimensional representation. 
Each $\sD_4\subset G$, gives a distinguished line and a point; together these form
a triangle.
We blow up the $3$ points to get a hexagon of lines, which form two
triples of disjoint lines, each line has faithful Klein $4$-group action and
generic stabilizer even $\bZ/2$.
The intersection points of the lines have stabilizer even $\mathfrak K_4$.
We blow up these intersection points.
The result is a wheel of $12$ rational curves:
\[
\xymatrix{   D_1 \ar@{-}[r] \ar@{-}[d] & R_1 \ar@{-}[r]  & D_1' \ar@{-}[r]  & R_2 \ar@{-}[r]  & D_2  \ar@{-}[r] &  R_3 \ar@{-}[d]  \\ 
R_6 \ar@{-}[r]& D_3'  \ar@{-}[r]   & R_5  \ar@{-}[r] & D_3 \ar@{-}[r] & R_4  \ar@{-}[r] &   D_2' 
}
\]
Each rational curve has generic stabilizer even $\bZ/2$, and their
intersection points have stabilizer even $\mathfrak K_4$.
The $12$ curves form $3$ orbits:
$\{ D_1,D_2,D_3\}$ and $\{ D_1',D_2',D_3'\}$ consist of
lines with generic stabilizer even $\bZ/2$ and faithful $\mathfrak K_4$-action, and 
the lines in the $G$-orbit $\{ R_1, \dots, R_6\}$ 
have generic stabilizer even $\bZ/2$ and a nontrivial $\bZ/2$-action.

The restriction of $V_3$ to $\fS_3$ also decomposes into a
$1$-dimensional and an irreducible $2$-dimensional representation. 
Looking at $G$-orbits, we find a $G$-orbit of $4$ distinguished  $\fS_3$-lines, which intersect in $6$ points with
odd $\mathfrak K_4$ stabilizer. We also have a $G$-orbit of
$4$ distinguished points with $\fS_3$-stabilizer.
We blow up the points with odd $\mathfrak K_4$-stabilizer and also those with $\fS_3$-stabilizer.
We get a $G$-orbit of $6$ lines, each with $\bZ/2$-action and
generic stabilizer odd $\bZ/2$, as well as an orbit of 4 exceptional curves with $\fS_3$-action. 

\

\noindent
\textbf{Second action:} 
Let $X$ be a del Pezzo surface of degree 6 given by 
$$
x_0y_0z_0=x_1y_1z_1 \subset \bP^1\times \bP^1\times \bP^1.
$$ 
We write the action of $G$ in coordinates 
$$
x:=x_0/x_1, \quad  y:=y_0/y_1, \quad z:=z_0/z_1.
$$ 
Then,
$\fS_3=\langle \sigma, \tau\rangle$ acts by permuting $x, y, z$;  $\lambda_1$ changes 
signs on $x$ and $z$, and $\lambda_2$ changes signs on  $x$ and $y$. 

There are three orbits of points, of length 4, with stabilizer $\fS_3$ (see \cite[Lemma 1.3]{BanTok}). 
Blowing these up, we obtain 3 $G$-orbits of $\fS_3$-lines. These 
do not contribute to $[X\actsfromright G]$. 
There are also two $G$-orbits of points with $\sD_4$-stabilizers, these points are precisely the intersection points of  
the 6 lines at infinity, i.e., in the locus 
$$
\{x_0=0\} \cup \{ y_0=0\} \cup \{z_0=0\}.
$$
These lines have generic stabilizer even $\bZ/2$ and a nontrivial $\bZ/2$-action, and they form a single $G$-orbit. 
After we blow up the two orbits of 3 points points, we obtain precisely the wheel configuration we described above.

\

To summarize, the difference 
\[
[X\actsfromright G] - [\bP^2\actsfromright G]
\]
is a symbol 
\begin{equation}
\label{eqn.difference}
(\text{odd }\bZ/2, \bZ/2 \actsfromleft k(t),(1))
\end{equation}
corresponding to a $G$-orbit of 6 lines with generic stabilizer odd $\bZ/2$ and nontrivial $\bZ/2$-action.
By a computation, analogous to the determination of $\Burn_2(\bZ/2\oplus \bZ/2)$ in
\cite[\S 5.4]{HKTsmall},
the class \eqref{eqn.difference} is nontrivial in $\Burn_2(G)$.

\bibliographystyle{plain}
\bibliography{struct}

\end{document}